\documentclass[reqno,12pt]{amsart}
\usepackage{amsmath,amssymb}
\usepackage{graphicx}
\usepackage[top=1in,bottom=1in,left=1in,right=1in]{geometry}
\newtheorem{thm}{Theorem}
\newtheorem{lem}{Lemma}[section]

\theoremstyle{definition}
\newtheorem{rem}{Remark}[section]
\numberwithin{equation}{section}

\usepackage{color}

\newcommand{\eps}{\varepsilon}

\newcommand{\be}{\begin{equation} \label}
\newcommand{\ee}{\end{equation}}
\newcommand{\bas}{\begin{eqnarray*}}
\newcommand{\eas}{\end{eqnarray*}}
\newcommand{\R}{\mathbb{R}}

\newcommand{\Rn}{\mathbb{R}^n}

\begin{document}

\title[refined blowup estimates]{On refined blowup estimates for the\\ exponential reaction-diffusion equation}

\author{Philippe Souplet}
\address{Universit\'e Paris 13, Sorbonne Paris Cit\'e, CNRS UMR 7539,
Laboratoire Analyse G\'eom\'etrie et Applications
93430 Villetaneuse, France.}
\email{souplet@math.univ-paris13.fr}

\begin{abstract}
We consider radial decreasing solutions
of the semilinear heat equation with exponential nonlinearity.
We provide a relatively simple proof of the sharp upper estimates for the final blowup profile and for the refined space-time behavior.
We actually establish a global, upper space-time estimate, which contains those of the final and refined profiles as special cases.


\vskip 2pt

\end{abstract}

\maketitle

\section{Introduction and main results}

\noindent

We consider the semilinear heat equation with exponential nonlinearity.
\be{QPEp2}
\left\{ \begin{array}{lllll}
          \hfill u_t-\Delta u&=&e^u, &\qquad& x\in\Omega,\ t>0,\\
     \hfill u  &= &0,  &\qquad& x\in\partial\Omega,\ t>0,\\
          \hfill u(x,0) &=&u_0(x), &\qquad& x\in\Omega,
	\end{array} \right.
\ee
where $\Omega\subset\R^n$, $n\ge 1$, and $u_0\in L^\infty(\Omega)$, $u_0\ge 0$.
Problem \eqref{QPEp2} has a unique classical solution, of maximal existence time $T=T(u_0)\in (0,\infty]$, that we will denote by $u$ throughout this paper.
It is well known that, under suitable largeness condition on the initial data,
$u$ blows up in finite time, i.e. $T<\infty$ and 
\be{defBU}
\lim_{t\to T} \|u(t)\|_\infty=\infty.
\ee

The asymptotic behavior of blowup solutions for problem \eqref{QPEp2} has received a lot of attention
(cf., e.g.,~\cite{La, FML85, Dol, BBE87, ET87, BE89, Bress90, Bress92, BB92, HV93, LT, hv-ijm93, fp_tmj08, pulk_mmas11}, 
      and see below for more details).
The main questions are the blow-up rate, the blow-up set and the asymptotic shape of the solution.
The latter, which involves various notions of blow-up profiles (final profile, self-similar profile, refined space-time profile),
is our concern in this paper.
\smallskip

Our first main result is the following global, refined blowup estimate, valid in the scale of the original variables $(x,t)$.
We shall assume
\be{hypu0}
\begin{aligned}
&\hbox{$\Omega=\R^n$ or $\Omega=B_R$, $u_0\in L^\infty(\Omega)$, $u_0\geq 0$,}\\
&\hbox{$u_0$ radially symmetric, nonincreasing in $r=|x|$, and nonconstant.}
\end{aligned}
\ee
Under this hypothesis, $u\ge 0$ is radially symmetric decreasing in $|x|$ for all $t\in(0,T)$,
hence in particular $u(0,t)=\|u(t)\|_\infty$.
Also it is well known \cite{FML85, Chen90} that $u$ can then blow up only at $x=0$.

\begin{thm} \label{thm1}
Assume \eqref{hypu0} and $T=T(u_0)<\infty$. Then, setting $m=m(t):=u(0,t)$, we have
\be{ineq1thm1}
u(r,t)\le \log\Bigl({|\log(me^{-m}+{r^2\over 4})|\over me^{-m}+{r^2\over 4}}\Bigr)+\eps(r,t)
\quad\hbox{ in $B_\rho\times (T-\rho,T)$,}
\ee
for some $\rho>0$ small, with $\displaystyle\lim_{ (r,t)\to (0,T)}\eps(r,t)=0$.
More precisely, we may take
\be{ineq2thm1}
\eps(r,t):={C\over |\log (me^{-m}+{r^2\over 4})|}+{C\log m\over m+{e^mr^2\over 4}}
\ee
for some constant $C>0$ (depending on $u$).
\end{thm}

\goodbreak

The global estimate \eqref{ineq1thm1} is new, as far as we know.
As special cases, it contains sharp upper estimates for the final profile\footnote{Recall that, under the assumptions of Theorem~\ref{thm1}, since $0$ is the only blowup point (this follows from \eqref{ineq1thm1}), the space-profile 
$u(x,T)= \lim_{t\to T}u(x,t)$ exists for all $x\ne 0$ by standard parabolic estimates.} 
and for the refined space-time profile.
Namely, as a consequence of Theorem~\ref{thm1}, we obtain:

\begin{thm} \label{thm2}
Under the assumptions of Theorem~\ref{thm1}, there exist constants $C,\rho>0$ such that the following holds.
\smallskip

(i) (Final profile estimate) 
\be{ineq1thm2}
u(r,T)\le 
2|\log r|+\log |\log r|+\log 8+{C\over |\log r|}\quad\hbox{ for $r\in(0,\rho]$.}
\ee

(ii) (Refined space-time profile estimate) 
\be{ineq2thm2}
u(\xi \sqrt{me^{-m}},t)\le m-\log\Bigl(1+{\xi^2\over 4}\Bigr)+{C\log m\over m}\quad\hbox{ for $t\in[T-\tau,T)$,}
\ee
where $\tau\in(0,T)$ is uniform for $\xi\ge 0$ bounded.
\end{thm}

Let us compare Theorems~\ref{thm1}--\ref{thm2} with known results.
Let $u$ be a (not necessarily radial) blow-up solution of \eqref{QPEp2}. 
First of all, concerning the blow-up rate, we always have
$$\|u(t)\|_\infty\ge -\log(T-t)$$
(see, e.g.,~\cite{fp_tmj08}), 
and blow-up is said to be of type~I if 
\be{type1def}
\|u(t)\|_\infty\le -\log(T-t)+C.
\ee
This is the case under any of the following assumptions:
\be{type1cond1}
\hbox{$\Omega$ is bounded and $u_t\geq0$,}
\ee
\be{type1cond2}
\hbox{$\Omega=\R^n$, $u_0$ is radial decreasing and either $n=1$ or $u_t\ge 0$,}
\ee
\be{type1cond3}
\hbox{$3\le n\le 9$, $\Omega=B_R$ and $u_0$ is radial decreasing}
\ee 
(see, respectively, \cite{FML85}, \cite{HV93, BB92} and \cite{fp_tmj08}).
It seems to be an open problem whether~\eqref{QPEp2} admits some type~II blowup solutions 
(i.e.,~violating~\eqref{type1def}).

Next, concerning the final and refined blow-up profiles, it is known that, under assumption \eqref{type1cond2}, there holds
\be{QprofExp}
u(x,T)\to 2|\log |x||+\log |\log |x||+\log 8, \quad x\to 0
\ee
and
\be{QprofExp2}
u\bigl(\xi \sqrt{(T-t)|\log(T-t)|},t\bigr)+\log(T-t)\to -\log\Bigl(1+{\xi^2\over 4}\Bigr),\quad t\to T,
\ee
uniformly for $\xi$ bounded (hence in particular $\lim_{t\to T} \,[u(0,t)+\log(T-t)]=0$); 
cf.~\cite{HV93} if $n=1$ and~\cite{BB92} if $u_t\ge 0$, and see also 
 \cite[Theorem~3.30]{BE89} and \cite[Theorem~3.1]{fp_tmj08} for related results.
Moreover~\cite{Bress90, Bress92}, for any convex domain $\Omega$, there exists an open set of initial data such that
\eqref{QprofExp}-\eqref{QprofExp2} is true (after a shift of the blow-up point).
We thus see that, whereas our results make no restriction on the space dimension nor require $u_t\ge 0$,
the upper bounds in Theorem~\ref{thm2} are sharp for such solutions.

However, \eqref{QprofExp}-\eqref{QprofExp2} is not the only possible behavior and other,
more or less singular final profiles may occur.
Indeed, the following classification result is proved in \cite{HV93}: if $\Omega=\R$, $u_0\ge 0$ is nonconstant
and $u$ blows up at $(x,t)=(0,T)$,
then we have either \eqref{QprofExp} or one of the more singular final profiles
\be{QprofExpm}
 u(x,T)+m\log|x|\to C_m\quad\hbox{as}\ x\to0,
\ee
for some integer $m\geq4$.
Moreover, there exists $u_0$ such that \eqref{QprofExpm}
with $m=4$ is true, see \cite{hv-ijm93} and the references therein.
Solutions satisfying \eqref{QprofExpm} cannot be symmetric decreasing, in view of the result in the previous paragraph
or of Theorem~\ref{thm2}.
On the other hand \cite{ET87, fp_tmj08, pulk_mmas11}, there exist solutions satisfying assumption \eqref{type1cond3},
with the less singular final profile
\be{QprofExp0}
u(x,T)+2\log|x|\to C\quad\hbox{as}\ x\to0.
\ee

\begin{rem} \label{rem0}
(i) We stress that the deep results reported above concerning the blow-up profiles \eqref{QprofExp}--\eqref{QprofExp0}
 were established 
by extremely long and delicate proofs relying, among many other things, on ideas from center manifold theory
applied to the equation rewritten in similarity variables. 
Although it of course gives only the upper part of the sharp estimates, and in a rather particular radial situation,
our proof is considerably simpler and shorter,
besides leading to the new global estimate \eqref{ineq1thm1} (and having a different range of applicability).

\smallskip

(ii) Concerning the final profile, the only available estimate on the rate of convergence in~\eqref{QprofExp}, as $x\to 0$,
seems to be that for the special solutions constructed in \cite{Bress92},
where the remainder is estimated by $C|\log |x||^{-1/3}$. 
We see that, under the assumptions of Theorem~\ref{thm2} and as far as the upper estimate is concerned,
we get a more precise remainder $C|\log r|^{-1}$ in \eqref{ineq1thm2}.
\end{rem}

\section{Proof of Theorems~\ref{thm1} and \ref{thm2}}%

The proof of Theorem~\ref{thm1} is based on a suitable 
modification of the method in \cite{FML85} (see also \cite{BE89, Sou}) applying the maximum principle
to a well-chosen auxiliary functional $J$ and carefully integrating the resulting differential inequality.
Namely, we shall consider
     \be{DefJ}
     J:=u_r+{re^u\over 2\bigl(A+u-\log(A+u)\bigr)},
\ee
where $A>0$ is a sufficiently large constant.

\medskip

{\it Proof of Theorem~\ref{thm1}.}
{\bf Step 1.} {\it Basic parabolic inequality.}
 This step is well known (see \cite{FML85, Sou}). We reproduce it for completeness.
We set $f(u):=e^u$ and $R=1$ in case $\Omega=\Rn$.
Since $u\ge e^{ t\Delta}u_0$, 
 where $(e^{t\Delta})_{t\ge 0}$ denotes the Dirichlet heat semigroup on $B_R$,  there exists $\eta>0$ such that
\be{lowerboundueta}
u(x,t)\ge \eta>0\quad\hbox{in $D:=\overline B_{R/2}\times [T/2,T)$}.
\ee
Setting $\Omega_1:=\Omega\cap\{x:\,x_1>0\}$, we notice that $v:=u_{x_1}\le 0$ satisfies $v_t-\Delta v=f'(u)v\le 0$ in 
$\Omega_1\times (0,T)$.
 Therefore, $v(t)\le z(t):=e^{(t-t_0)A}v(\cdot,t_0)$ in $\Omega_1\times (t_0,T)$, where $t_0:=T/4$ and $e^{tA}$ denotes
the Dirichlet heat semigroup on $\Omega_1$.
It follows from the strong maximum principle and the Hopf Lemma, applied to $z$, that 
$v(x_1,0,\dots,0,t)\le -k x_1$ for all $(x_1,t)\in [0,R/2]\times [T/2,T)$ and some $k>0$.  
This yields
\be{FMLHopf}
u_r\le -k r\quad\hbox{ in $[0,R/2]\times [T/2,T)$.}
\ee

We next consider an auxiliary function of the form  $J:=u_r(r,t)+c(r)F(u)$,
where the functions $c$ and $F$ will be chosen below.
In $Q:=(0,R/2)\times (T/2,T)$, we compute
$$\Bigl({\partial\over\partial t}-{\partial^2\over\partial r^2}\Bigr)(cF(u))=cF'(u)(u_t-u_{rr})-cF''(u)u_r^2-2c'F'(u)u_r-c''F(u)$$
and 
$$\Bigl({\partial\over\partial t}-{\partial^2\over\partial r^2}\Bigr)u_r={n-1\over r}u_{rr}-{n-1\over r^2}u_{r}+f'(u)u_r.$$
Omitting the variables $r,t,u$ without risk of confusion, it follows that
$$J_t-J_{rr}=
{n-1\over r}u_{rr}-{n-1\over r^2}u_{r}+f'u_r
+cF'\Bigl({n-1\over r}u_{r}+f\Bigr)-cF''u_r^2-2c'F'u_r-c''F.
$$
Substituting $u_r=J-cF$ and $u_{rr}=J_r-c'F-cF'u_r=J_r-cF'J+c^2FF'-c'F$, we obtain
$$\begin{aligned}
J_t-J_{rr}
&={n-1\over r}(J_r-cF'J+c^2FF'-c'F)-{n-1\over r^2}(J-cF)+f'(J-cF) \\
&\quad+cF'\Bigl({n-1\over r}(J-cF)+f\Bigr)-cF''(J-cF)^2-2c'F'(J-cF)-c''F.
\end{aligned}$$
Setting
\be{defPJ}
\mathcal{P}J:=J_t-J_{rr}-{n-1\over r}J_{r}+bJ,\quad\hbox{ with } b:={n-1\over r^2}-f' +cF''(J -2cF)+2c'F',
\ee
it follows that
$$\begin{aligned}
\mathcal{P}J
&={n-1\over r}(c^2FF'-c'F)+{n-1\over r^2}cF-cFf'\\
&\qquad +cF'\Bigl(-{n-1\over r}cF+f\Bigr)-c^3F''F^2+2cc'FF'-c''F,
\end{aligned}$$
hence
\be{defPJ2}
\mathcal{P}J=c(F'f-Ff')+{n-1\over r^2}(c-rc')F-c^3F''F^2+2cc'FF'-c''F.
\ee

{\bf Step 2.} {\it Choice of auxiliary functions.}
Now choose $c(r)=r/2$, hence $c-rc'=c''=0$, and $F(u)=f(u)\phi(u)$,
where $f(u):=e^u$ and the function $\phi\in C^2([0,\infty))$, to be determined, 
satisfies 
     \be{QFMLJH2}
     \phi>0,\quad \phi'\le 0,\quad (f\phi)''\ge 0\quad\hbox{in $[0,\infty)$.}
\ee
We have in $Q$:
$$\begin{aligned}
c^{-1}f^{-2}\mathcal{P}J
&\le f^{-2}(F'f-Ff'+F'F) = f^{-1}(F'-F+F'\phi) \\
&=f^{-1}\bigl[f\phi'+\phi(f'\phi+f\phi')\bigr] = \phi'+\phi(\phi+\phi').
\end{aligned}$$
To guarantee 
$\mathcal{P}J\le 0$, 
we then select
     \be{choicephi}
     \phi(s)={1\over A+s-\log(A+s)},\quad s\ge 0,
     \ee
where the constant $A>1$ will be chosen below. Indeed, for all $s\ge 0$, we have $A+s-\log(A+s)\ge 1$ and 
$$\begin{aligned}{}
[\phi'+\phi(\phi&+\phi')](s) \\
&={-1+(A+s)^{-1}\over (A+s-\log(A+s))^2}\Bigl[1+{1\over A+s-\log(A+s)}\Bigr]+{1\over (A+s-\log(A+s))^2} \\
&={(A+s)^{-1}\over (A+s-\log(A+s))^2}+{-1+(A+s)^{-1}\over (A+s-\log(A+s))^3} \\
&={(A+s)^{-1}\bigl(A+s-\log(A+s)\bigr)-1+(A+s)^{-1}\over (A+s-\log(A+s))^3} \\
&={(A+s)^{-1}\bigl(1-\log(A+s)\bigr)\over (A+s-\log(A+s))^3} \le 0.
\end{aligned}$$
Moreover, an elementary computation shows that \eqref{QFMLJH2} is true for all $A>0$ sufficiently large.
On the other hand, assuming $A\ge 2$, we get 
$2(A+s-\log(A+s))\ge A$ for all $s\ge 0$ (since $z\ge 2\log z$ for all $z\ge 2$).
It thus follows from \eqref{FMLHopf} that
$$r^{-1}J\le - k+A^{-1}e^u\quad\hbox{ in }(0,R/2]\times[T/2,T).$$
Since we know that $0$ is the unique blowup point 
(see \cite{FML85} for $\Omega=B_R$ and \cite{Chen90} for $\Omega=\R^n$),
we may choose $A$ sufficiently large so that $J\le 0$ on the parabolic boundary of $Q$.
Since the coefficient $b$ in \eqref{defPJ} is bounded from below for $t$ bounded away from $T$,
we may apply the maximum principle to deduce that $J\le 0$, i.e.
\be{integr0}
        -e^{-u}(A+u-\log(A+u))u_r \ge {r\over 2}. 
\ee

In view of integrating inequality \eqref{integr0}, we give the following lemma.
 
\begin{lem}\label{leminverse}
Let $A>1$, $D=A+1-\log A>0$ and
$$G(u):=e^{-u}(A+1+u-\log(A+u)).$$
Then $G:[0,\infty)\to (0,D]$ is a decreasing bijection and there exist $C>0$, $s_0\in(0,\frac12)$ such that
\be{integr1a}
G^{-1}(s)\le H(s):=-\log s+\log|\log s|+{C\over |\log s|},\quad 0<s<s_0.
\ee
\end{lem}

\begin{proof} We have $G'(u):=-e^{-u}((A+u)^{-1}+A+u-\log(A+u))<0$,
hence the first assertion.
To show \eqref{integr1a}, substituting $s=G(u)$,
it is sufficient to check that, for some $C>0$, 
\be{integr1}
H(G(u))\ge u,\quad\hbox{ for all $u$ sufficiently large.}
\ee
To this end, setting $B=A+1$, we compute
$$-\log\bigl(e^{-u}(B+u-\log(A+u))\bigr)
=u-\log u-\log\bigl(1+(B-\log(A+u))u^{-1}\bigr)$$
and 
$$\begin{aligned}
\log\bigl|\log\bigl(e^{-u}(B+u-\log(A+u))\bigr)\bigr|
&=\log\bigl[u-\log(B+u-\log(A+u))\bigr]\\
&=\log u+\log\bigl[1-u^{-1}\log(B+u-\log(A+u))\bigr].
\end{aligned}$$
Therefore, for $u$ large,
$$\begin{aligned}
H(G(u)&)-u \\
&=-\log\Bigl(1+{B-\log(A+u)\over u}\Bigr)+\log\Bigl[1-{\log(B+u-\log(A+u))\over u}\Bigr]+{C\over |\log G(u)|} \\
&={\log(A+u)-B\over u}-{\log(B+u-\log(A+u))\over u}+{C\over u+O(\log u)}+O(u^{-2}|\log u|^2) \\
&={\log u-B+O(u^{-1})\over u}-{\log u+O(u^{-1}\log u))\over u}+{C\over u}+O(u^{-2}|\log u|^2) \\
&={C-B\over u}+O(u^{-2}|\log u|^2).
\end{aligned}$$
We deduce that \eqref{integr1} is true with $C=2B$.
\end{proof}

{\bf Step 3.} {\it Integration.}
Integrating \eqref{integr0} by parts, we obtain, in $Q$:
$$\begin{aligned}
 {r^2\over 4}&\le\int_{u(r,t)}^{u(0,t)} e^{-z}(A+z-\log(A+z))\, dz \\
&=\Bigl[-e^{-z}(A+z-\log(A+z))\Bigr]_{u(r,t)}^{u(0,t)}+\int_{u(r,t)}^{u(0,t)} e^{-z}(1-(A+z)^{-1})\, dz \\
&\le\Bigl[-e^{-z}(A+1+z-\log(A+z))\Bigr]_{u(r,t)}^{u(0,t)}
\end{aligned}$$
hence, recalling $m=m(t)=u(0,t)$,
$$G(u(r,t))=e^{-u}(A+1+u-\log(A+u))\ge 
s=s(r,t):=e^{-m}(m-\log(A+m))+{r^2\over 4}.$$
Now, for $(r,t)$ close to $(0,T)$, we have $s\in(0,s_0)$ owing to \eqref{defBU}, so that
Lemma~\ref{leminverse} yields
\be{integr2}
u(r,t)\le -\log s+\log|\log s|+{C\over |\log s|}.
\ee
On the other hand, for $(r,t)$ close to $(0,T)$, we have $m>\log(A+m)$ and
$$\begin{aligned}
|\log s|
&=-\log\Bigl(me^{-m}+{r^2\over 4}\Bigr)-\log\Bigl(1-{e^{-m}\log(A+m)\over me^{-m}+{r^2\over 4}}\Bigr) \\
&=-\log\Bigl(me^{-m}+{r^2\over 4}\Bigr)-\log\Bigl(1-{\log(A+m)\over m+{e^mr^2\over 4}}\Bigr) \\
&=-\log\Bigl(me^{-m}+{r^2\over 4}\Bigr)+O\biggl({\log m\over m+{e^mr^2\over 4}}\biggr)
\end{aligned}$$
and
$$\begin{aligned}
\log|\log s|
&=\log\biggl\{\Bigl|\log\Bigl(me^{-m}+{r^2\over 4}\Bigr)\Bigr|+O\biggl({\log m\over m+{e^mr^2\over 4}}\biggr)\biggr\} \\
&=\log\Bigl|\log\Bigl(me^{-m}+{r^2\over 4}\Bigr)\Bigr|+O\biggl({\log m\over (m+{e^mr^2\over 4})|\log(me^{-m}+{r^2\over 4})|}\biggr).
\end{aligned}$$
This, along with \eqref{integr2}, guarantees \eqref{ineq1thm1}-\eqref{ineq2thm1}.
\qed

\begin{proof}[Proof of Theorem~\ref{thm2}]
(i) Letting $m\to\infty$ in \eqref{ineq1thm1}-\eqref{ineq2thm1}, we get
$$
u(r,T)\le \log\Bigl({|\log({r^2\over 4})|\over {r^2\over 4}}\Bigr)+{C\over |\log ({r^2\over 4})|}
\le 2|\log r|+\log 4+\log\bigl|2\log r-\log 4\bigr|+{C\over |\log r|},$$
hence \eqref{ineq1thm2}.

\smallskip
(ii) Fix $K>0$. By \eqref{defBU}, for $\tau>0$ small (depending on $K$) we have $\log\bigl(m(1+{K^2\over 4})\bigr)<m/2$ for all $t\in[T-\tau,T)$,
with also $K\sqrt{me^{-m}}<R$ if $\Omega=B_R$.
For any $(\xi,t)\in [0,K]\times[T-\tau,T)$, letting $r=\xi \sqrt{me^{-m}}$ in \eqref{ineq1thm1}-\eqref{ineq2thm1}, we get
$$\begin{aligned}
&u\bigl(\xi \sqrt{me^{-m}},t\bigr)\\
&\le -\log\Bigl(me^{-m}\Bigl(1+{\xi^2\over 4}\Bigr)\Bigr)
+\log\Bigl|\log\Bigl(me^{-m}\Bigl(1+{\xi^2\over 4}\Bigr)\Bigr)\Bigr|+
{C\over \bigl|\log \bigl(me^{-m}\bigl(1+{\xi^2\over 4}\bigr)\bigr)\bigr|}+{C\log m\over m} \\
&\le m-\log m-\log\Bigl(1+{\xi^2\over 4}\Bigr)+\log\Bigl[m-\log\Bigl(m\Bigl(1+{\xi^2\over 4}\Bigr)\Bigr)\Bigr]
+{2C\over m} +{C\log m\over m} \\
&\le m-\log\Bigl(1+{\xi^2\over 4}\Bigr)+{\tilde C\log m\over m},
\end{aligned}$$
which is the desired conclusion.
\end{proof}

\begin{rem} \label{rem1}
(i) In the seminal work \cite{FML85}, the functional $J:=u_r+\eps re^u$ (with $\eps>0$ small) was used, instead of~\eqref{DefJ},
leading to the estimate $u(r,T)\le (2+\eta)|\log r|+C_\eta$ (for all $\eta>0$).

\smallskip

(ii) In \cite{BE89}, the functional $J:=u_r+{\eps re^u\over 2+u}$ (with $\eps>0$ small) was used
(see \cite[Corollary~3.17]{BE89}).
This gave the final profile estimate
$u(r,T)\le 2|\log r|+\log |\log r|+C$ with a large (unspecified) constant $C>0$, 
instead of the sharp constant~$\log 8$ plus a remainder term.

\smallskip

(iii) In these works, the possibility to establish sharp space-time estimates such as \eqref{ineq1thm1} and \eqref{ineq2thm2} 
by this method was not considered.
In the case of the reaction-diffusion equation with power nonlinearity,
results related to Theorems~\ref{thm1}--\ref{thm2} and obtained by a similar method
can be found in \cite{Sou} (and we refer to \cite{QSb} for more references on that problem).

\end{rem}

\begin{rem} \label{rem2}
(i) In the proof of Theorem~\ref{thm1}, the simpler choice 
$$
\phi(u)={1\over A+u},
$$
 instead of \eqref{choicephi}, i.e.~$J:=u_r+{re^u\over 2(A+u)}$ (with $A$ large),
 is also possible.
However it gives \eqref{ineq1thm1}-\eqref{ineq2thm1} with $\eps$ replaced by the less precise remainder term 
${C\log|\log (me^{-m}+{r^2\over 4})|\over |\log (me^{-m}+{r^2\over 4})|}$,
and \eqref{ineq1thm2} with ${C\log|\log r|\over |\log r|}$
instead of ${C\over |\log r|}$.
\smallskip

(ii) Instead of \eqref{choicephi},
the apparently optimal choice of $\phi$ would be the (nonexplicit) solution of the ODE $\phi'=-(1+\phi)^{-1}\phi^2$.
However one can check that this does not produce any improvement of
the remainder term in \eqref{ineq2thm1}.

\smallskip

(iii) The term $\tilde T:=-c^3F''F^2$ in \eqref{defPJ2} is not used (except for its sign) and it does not seem possible to exploit it
in order to improve the remainder term in \eqref{ineq2thm1}.
In other situations, such as the diffusion equation with fast absorption (cf.~\cite{GS05} and see also 
\cite[Section~38]{QSb}) or the parabolic-elliptic Keller-Segel problem \cite{SWi}, the term corresponding to $\tilde T$
can be used, actually in a crucial way, through a nonlocal version of the maximum principle,
but this requires to exploit the different, and rather specific structure of the equation.
\end{rem}

\goodbreak

\end{document}